\documentclass[12pt]{amsart}

\usepackage{graphicx,amsmath,mathtools,amssymb}
\usepackage{epsfig,color,hyperref,fancyhdr,geometry}
\usepackage{dsfont}
\usepackage{epstopdf}

\newtheorem{theorem}{Theorem}[section]
\newtheorem{lemma}[theorem]{Lemma}
\newtheorem{definition}[theorem]{Definition}

\newtheorem{proposition}[theorem]{Proposition}
\newtheorem{remark}[theorem]{Remark}
\newtheorem{corollary}[theorem]{Corollary}
\newtheorem{conjecture}[theorem]{Conjecture}

\newtheorem{counterexample}[theorem]{Counterexample}

\numberwithin{equation}{section}

\begin{document}

\title{The second Yang-Baxter homology for the Homflypt polynomial}

\author{J\'{o}zef H. Przytycki}

\author{Xiao Wang}

\begin{abstract}

In this article, we adjust the Yang-Baxter operators constructed by Jones for the HOMFLYPT polynomal.  Then we compute the second homology for this family of Yang-Baxter operators.  It has the potential to yield 2-cocycle invariant for links. 

\end{abstract}

\subjclass[2020]{Primary: 57K10. Secondary: 16Txx.}

\maketitle

\tableofcontents

\section{Introduction}

The Yang-Baxter equation was first introduced independently in a study of many body quantum system by Yang \cite{Yan} and statistical mechanics by Baxter \cite{Bax} \footnote{The name Yang-Baxter equation was coined by Ludvig Faddeev}.  Since the discovery of the Jones polynomial \cite{Jon} in 1984, solutions to Yang-Baxter equation have become important for knot theory.  In particular, Jones \cite{Jon} and Turaev \cite{Tur} built a machinery to construct link invariants using Yang-Baxter operators and the family of Yang-Baxter operators from the representation of $A^{1}$ series lead to $sl(m)$ polynomial invariants whose ``limit" is the Homflypt polynomial \cite{Homfly,PT}. Racks and quandles give special examples of Yang-Baxter operators. Homology theory of rack and quandles were introduced in \cite{FRS-1,CJKLS}.  Carter, Elhamdadi and Saito \cite{CES} defined a (co)homology theory for set-theoretic Yang-Baxter operators generalizing this homology, and they constructed cocycle link invariants. The homology theory of general Yang-Baxter operators were developed by Lebed \cite{Leb} and Przytycki \cite{Prz} independently and this homology theory is equivalent to the one defined in \cite{CES} when restricted to the set-theoretic Yang-Baxter operators \cite{PW}.  In the first part of this paper, we give a detailed proof that after modifying the Yang-Baxter matrix obtained from $A_{1}$ series to be column unital, they are still Yang-Baxter operators.  Furthermore, this new family of operators also lead to the $sl(m)$ polynomial invariants \cite{Wan}, which is implicit in \cite{Jon}.  The homology can be defined for any column unital Yang-Baxter operators, see \cite{Prz}. In the second part of the paper, we compute the second homology of the column unital Yang-Baxter operators corresponding to $sl_{m}$ link invariants denoted by $R_{(m)}$(see Theorem \ref{Theorem 5.1}).

\section{Column unital Yang-Baxter operators}

Inspired by statistical mechanics, Jones constructed the Yang-Baxter operators leading to the Jones and HOMFLYPT polynomials, see \cite{Jon2, Tur} for more information on the use of Yang-Baxter operators in knot theory. 
\begin{definition}\label{Definition 1.1}
Let $k$ be a commutative ring and $V$ be a $k-$module.  If a $k-$linear map, $R:$ $V\otimes V \to V\otimes V$, satisfies the following equation called Yang-Baxter equation\\
$$(R\otimes Id_{V})\circ (Id_{V}\otimes R)\circ (R\otimes Id_{V})=(Id_{V}\otimes R)\circ (R\otimes Id_{V})\circ (Id_{V}\otimes R),$$\ \\
then we say $R$ is a pre-Yang-Baxter operator. If, in addition, $R$ is invertible, then we say $R$ is a Yang-Baxter operator.
\end{definition}
 Jones' Yang-Baxter operator on level $m$ is given by the following formula, 
$$R^{a b}_{c d}=\left\{
                  \begin{array}{ll}
                    -q, & \hbox{if a=b=c=d;} \\
                    1, & \hbox{if d=a $\neq$ b=c;} \\
                    q^{-1}-q, & \hbox{if c=a$<$b=d;} \\
                    0, & \hbox{otherwise.}
                  \end{array}
                \right.$$
           where $1\leq a,b,c,d \leq m$     
                
In this section, we give a detailed proof that the family of column unital operators defined in Theorem \ref{Column unital} are Yang-Baxter operators.  These operators are obtained from the Jones' Yang-Baxter operators by dividing each column by the sum of elements in the column and substitution $y^{2}=\frac{1}{1+q^{-1}-q}$.

\begin{theorem}\label{Column unital}

Let $k=\mathbb{Z}[y,y^{-1}]$, $m$ be a positive integer, and $V_{m}$ be the free $k$ module generated by the set $X_{m}=\{v_{1},...,v_{m}\}$ with ordering $v_{a}\leq v_{b}$ if and only if $a\leq b$.  Then the $k$ linear operator $R_{(m)}:V_{m}\otimes V_{m}\rightarrow V_{m}\otimes V_{m}$ given by the coefficients 

$$R^{a b}_{c d}=\left\{
                  \begin{array}{ll}
                    1, & \hbox{if d=a$\geq $b=c;} \\
                    y^2, & \hbox{if d=a$<$b=c;} \\
                    1-y^2, & \hbox{if c=a$<$b=d;} \\
                    0, & \hbox{otherwise.}
                  \end{array}
                \right.$$
               
is a Yang-Baxter operator for each $m\geq 1$.

\end{theorem}

One can check directly that the inverse of the these operators is 

$$(R^{-1})^{a b}_{c d}=\left\{
                  \begin{array}{ll}
                    1, & \hbox{if d=a$\leq $b=c;} \\
                    y^{-2}, & \hbox{if d=a$>$b=c;} \\
                    1-y^{-2}, & \hbox{if c=a$>$b=d;} \\
                    0, & \hbox{otherwise.}
                  \end{array}
                \right.$$

Before the proof of Theorem \ref{Column unital}, we set up some notations and give Proposition \ref{leq}. Throughout the paper, we will write $R,V,X$ for $R_{(m)},V_{(m)},X_{(m)}$ defined in Theorem \ref{Column unital}, respectively. In any statement, whenever we use $R,V,X$, it implies the statement is true for $R_{(m)},V_{(m)},X_{(m)}$, $\forall m=2,3,...$. We will use integers $1\leq a,b,c\leq m$ to represent the basis elements $v_{a},v_{b},v_{c}$ and $(a,b,c)$ for the tensor product $v_{a}\otimes v_{b}\otimes v_{c}$. 

\begin{proposition}\label{leq}

$R(a,a)=(a,a)$ agrees with the formulas $R(a,b)=(1-y^{2})(a,b)+y^{2}(b,a)$ when $a<b$, $R(a,b)=(b,a)$ when $a>b$ by substituting $b=a$.

\end{proposition}

\begin{proof}

 $R(a,a)=(a,a)=(1-y^{2})(a,a)+y^{2}(a,a)$.

\end{proof}

Now, we prove Theorem \ref{Column unital}

\begin{proof}

For $m=1$, the Yang-Baxter equation hold trivially.

We consider the cases of $m\geq 2$.

Let $a\leq b\leq c$ for $a,b,c\in X_{(m)}$, then by Proposition \ref{leq}, we need to check in total six cases for the Yang-Baxter equation, which are $(a,b,c);(b,a,c);(a,c,b);(b,c,a);(c,a,b);(c,b,a)\in X_{(m)}^{3}$.   We start from the case of $(a,b,c)$.
From the left-hand-side of the Yang-Baxter equation, computing terms by terms, we get the following (see Figure \ref{al}),

\begin{figure}

\includegraphics[scale=0.19]{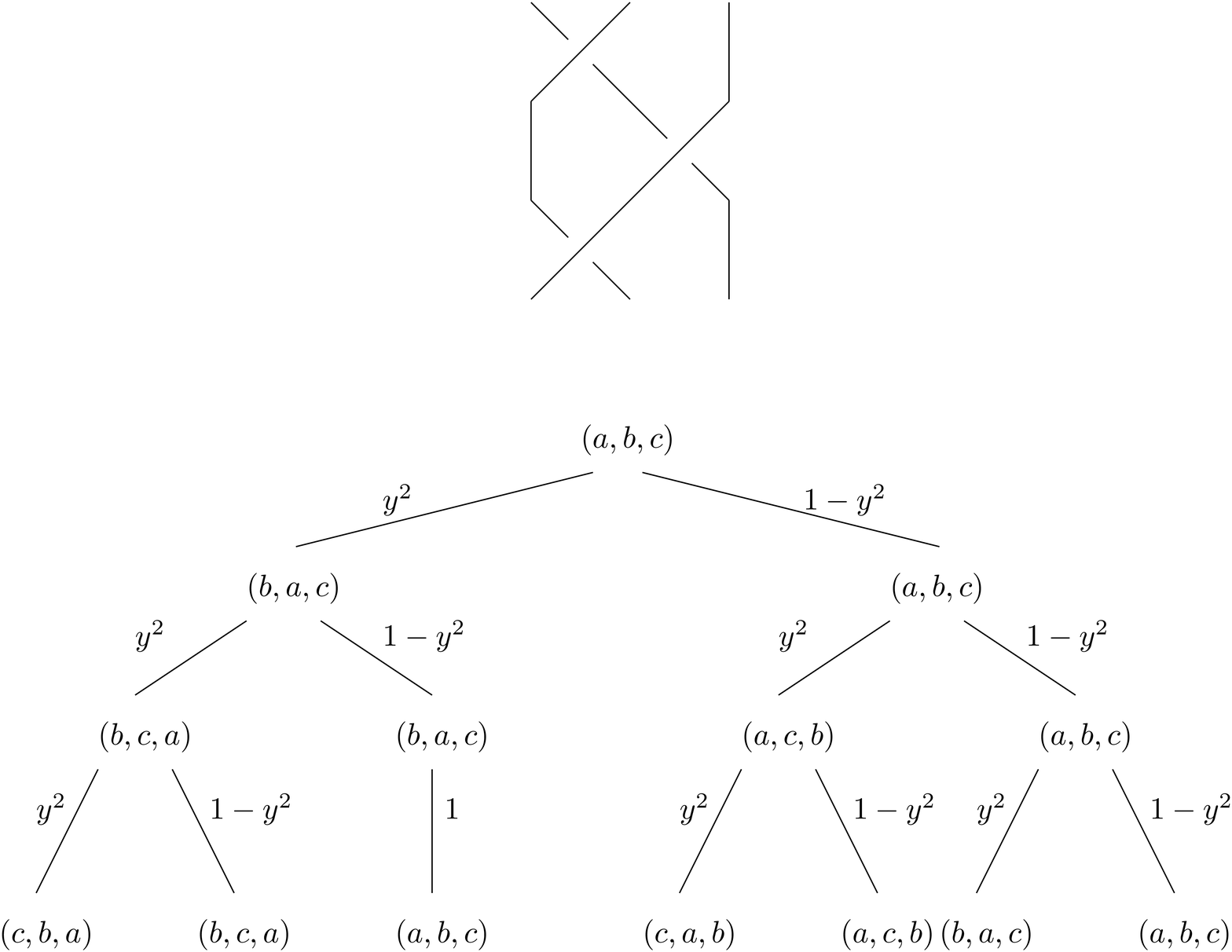} 

\caption{Computational tree for the left-hand-side of the Yang-Baxter equation of $(a,b,c)$}\label{al}

\end{figure}

$$(R\otimes Id_{V})\circ (Id_{V}\otimes R)\circ (R\otimes Id_{V})(a,b,c)=(R\otimes Id_{V})\circ (Id_{V}\otimes R)(y^{2}(b,a,c)+(1-y^{2})(a,b,c)))$$

$$(R\otimes Id_{V})\circ (Id_{V}\otimes R)(y^{2}(b,a,c))=(R\otimes Id_{V})(y^{2}y^{2}(b,c,a)+(1-y^{2})y^{2}(b,a,c))$$

$$((R\otimes Id_{V})(y^{2}y^{2}(b,c,a))=y^{2}y^{2}y^{2}(c,b,a)+(1-y^{2})y^{2}y^{2}(b,c,a)$$

$$((R\otimes Id_{V})((1-y^{2})y^{2}(b,a,c))=(1-y^{2})y^{2}(a,b,c),$$
and 

$$(R\otimes Id_{V})\circ (Id_{V}\otimes R)((1-y^{2})(a,b,c))=(R\otimes Id_{V})(y^{2}(1-y^{2})(a,c,b)+(1-y^{2})(1-y^{2}((a,b,c))$$

$$(R\otimes Id_{V})(y^{2}(1-y^{2})(a,c,b))=y^{2}y^{2}(1-y^{2})(c,a,b)+(1-y^{2})y^{2}(1-y^{2})(a,c,b)$$

$$(R\otimes Id_{V})((1-y^{2})(1-y^{2}(a,b,c))=y^{2}(1-y^{2})(1-y^{2}(b,a,c)+(1-y^{2})(1-y^{2})(1-y^{2}(a,b,c)$$

Similarly, we deal with the right-hand-side of the Yang-Baxter equation, computing terms by terms, we get the following (see Figure \ref{ar}),

\begin{figure}

\includegraphics[scale=0.19]{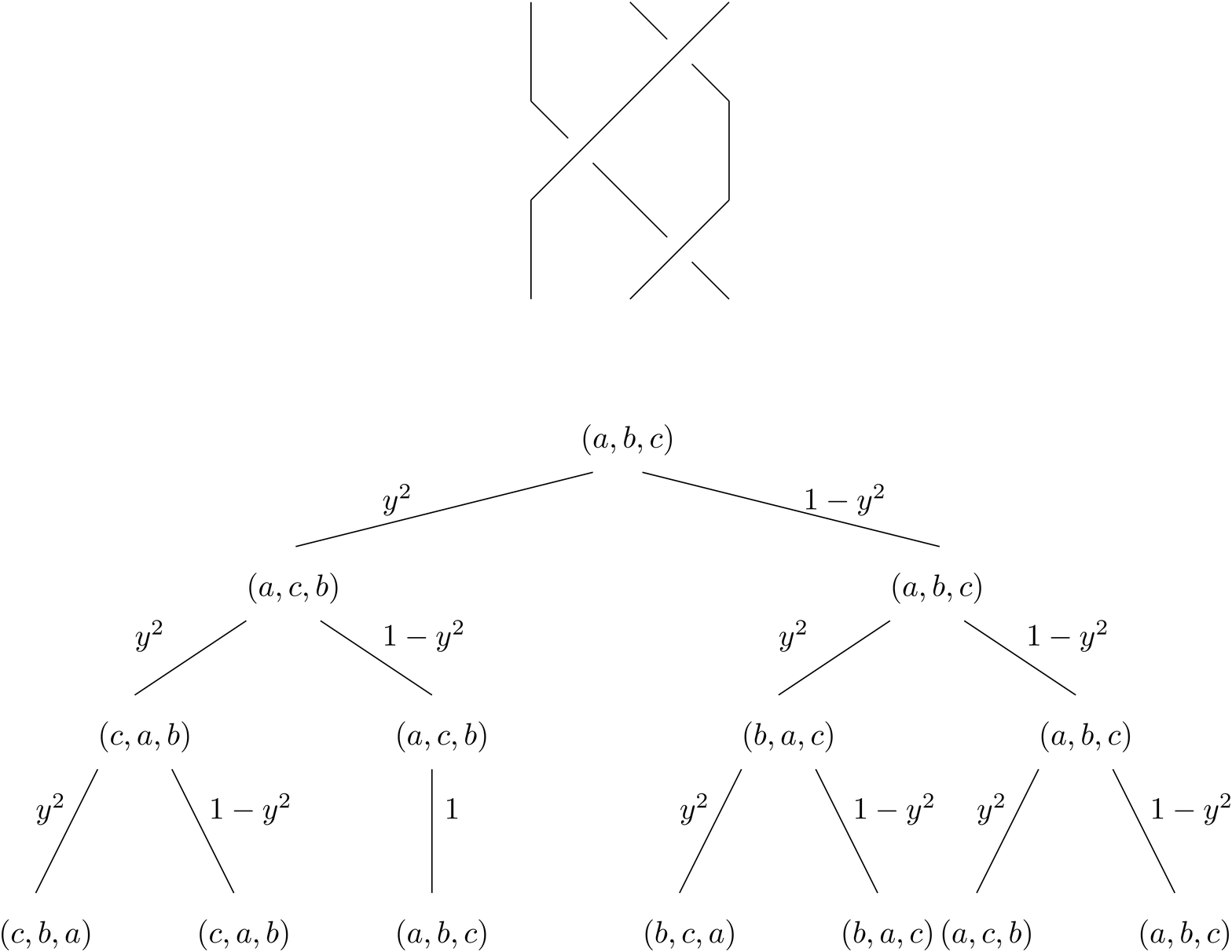} 

\caption{Computational tree for the right-hand-side of the Yang-Baxter equation of $(a,b,c)$}\label{ar}

\end{figure}

$(Id_{V}\otimes R)\circ (R\otimes Id_{V})\circ (Id_{V}\otimes R)(a,b,c)=(Id_{V}\otimes R)\circ (R\otimes Id_{V})(y^{2}(a,c,b)+(1-y^{2}(a,b,c)))$

$$(Id_{V}\otimes R)\circ (R\otimes Id_{V})(y^{2}(a,c,b))=Id_{V}\otimes R)(y^{2}y^{2}(c,a,b)+(1-y^{2})y^{2}(a,c,b))$$

$$(Id_{V}\otimes R)(y^{2}y^{2}(c,a,b))=y^{2}y^{2}y^{2}(c,b,a)+(1-y^{2})(1-y^{2})y^{2}(c,a,b)$$

$$(Id_{V}\otimes R)((1-y^{2})y^{2}(a,c,b))=(1-y^{2})y^{2}(a,b,c),$$

and

$$(Id_{V}\otimes R)\circ (R\otimes Id_{V})((1-y^{2}(a,b,c))=Id_{V}\otimes R)(y^{2}(1-y^{2}(b,a,c)+(1-y^{2})(1-y^{2}(a,b,c))$$

$$Id_{V}\otimes R)(y^{2}(1-y^{2}(b,a,c))=y^{2}y^{2}(1-y^{2}(b,c,a)+(1-y^{2})y^{2}(1-y^{2}(b,a,c)$$

$$Id_{V}\otimes R)((1-y^{2})(1-y^{2}(a,b,c))=y^{2}(1-y^{2})(1-y^{2}(a,c,b)+(1-y^{2})(1-y^{2})(1-y^{2}(a,b,c)$$

Both expressions are equal, thus prove Yang-Baxter equation holds for $(a,b,c)$. The other cases can be checked directly in a similar way. 

\end{proof}

\begin{remark}

From our proof and Figure \ref{al} and Figure \ref{ar}, we can conclude more.
$$(R\otimes Id_{V})\circ (Id_{V}\otimes R)\circ (R\otimes Id_{V})(a,b,c)=$$ 
$$[[y^{2}y^{2}y^{2}(c,b,a)+(1-y^{2})y^{2}y^{2}(b,c,a)]+[(1-y^{2})y^{2}(a,b,c)]]+$$ 
$$[[y^{2}y^{2}(1-y^{2})(c,a,b)+(1-y^{2})y^{2}(1-y^{2})(a,c,b)]+[y^{2}(1-y^{2})(1-y^{2}(b,a,c)+(1-y^{2})(1-y^{2})(1-y^{2}(a,b,c)]]$$

$$(Id_{V}\otimes R)\circ (R\otimes Id_{V})\circ (Id_{V}\otimes R)(a,b,c)=$$ 
$$[[y^{2}y^{2}y^{2}(c,b,a)+(1-y^{2})(1-y^{2})y^{2}(c,a,b)]+[(1-y^{2})y^{2}(a,b,c)]]+$$ 
$$[[y^{2}y^{2}(1-y^{2}(b,c,a)+(1-y^{2})y^{2}(1-y^{2}(b,a,c)]+[y^{2}(1-y^{2})(1-y^{2}(a,c,b)+(1-y^{2})(1-y^{2})(1-y^{2}(a,b,c)]]$$

Terms in the sum correspond to the leaves of the computational tree.  Square brackets group terms according the structure of the tree (see Figure \ref{al} and Figure \ref{ar}).

Important observation is that if we transform the result of the left-hand-side of $(a,b,c)$ by first switching the position of $a$ and $c$ and then reversing the order of the triple, we obtain exactly the result of right-hand-side of $(a,b,c)$ square bracket-wisely.  This observation can actually reduce the number of cases to check, which is important for us to compute the higher level homology in the future.  

\end{remark}

As mentioned before, the family of Yang-Baxter operators, $R_{(m)}$, have the property that summation of elements in each column of the matrix presentation equals to $1$. They are obtained from the Yang-Baxter operators leading to the Jones and HOMFLYPT polynomials \cite{Jon,Tur} by normalizing each column. 
However, normalizing columns of Yang-Baxter operators does not always produce Yang-Baxter operators in general.  

\begin{counterexample}

The following Yang-Baxter operator leading to the Kauffman two-variable polynomial (see \cite{Tur} for detail) with substitution $m=4$, $\nu =-1$ is a counterexample.

\begin{equation*}       
\begin{tiny}
\left(                 
  \begin{array}{cccccccccccccccc}   
 
 q & 0 & 0 & 0 & 0 & 0 & 0 & 0 & 0 & 0 & 0 & 0 & 0 & 0 & 0 & 0\\
 0 & q - q^{-1} & 0 & 0 & 1 & 0 & 0 & 0 & 0 & 0 & 0 & 0 & 0 & 0 & 0 & 0\\
 0 & 0 & q - q^{-1} & 0 & 0 & 0 & 0 & 0 & 1 & 0 & 0 & 0 & 0 & 0 & 0 & 0\\
 0 & 0 & 0 & q - 2 q^{-1} + (q^{-3}) & 0 & 0 & (q^{-2} - 1) & 0 & 0 & (q^{-2} - 1) & 0 & 0 & q^{-1} & 0 & 0 & 0\\
 0 & 1 & 0 & 0 & 0 & 0 & 0 & 0 & 0 & 0 & 0 & 0 & 0 & 0 & 0 & 0\\
 0 & 0 & 0 & 0 & 0 & q & 0 & 0 & 0 & 0 & 0 & 0 & 0 & 0 & 0 & 0\\
 0 & 0 & 0 & (q^{-2} - 1) & 0 & 0 & 0 & 0 & 0 & q^{-1} & 0 & 0 & 0 & 0 & 0 & 0\\
 0 & 0 & 0 & 0 & 0 & 0 & 0 & q - q^{-1} & 0 & 0 & 0 & 0 & 0 & 1 & 0 & 0\\
 0 & 0 & 1 & 0 & 0 & 0 & 0 & 0 & 0 & 0 & 0 & 0 & 0 & 0 & 0 & 0\\
 0 & 0 & 0 & (q^{-2} - 1) & 0 & 0 & q^{-1} & 0 & 0 & 0 & 0 & 0 & 0 & 0 & 0 & 0\\
 0 & 0 & 0 & 0 & 0 & 0 & 0 & 0 & 0 & 0 & q & 0 & 0 & 0 & 0 & 0\\
 0 & 0 & 0 & 0 & 0 & 0 & 0 & 0 & 0 & 0 & 0 & q - q^{-1} & 0 & 0 & 1 & 0\\
 0 & 0 & 0 & q^{-1} & 0 & 0 & 0 & 0 & 0 & 0 & 0 & 0 & 0 & 0 & 0 & 0\\
 0 & 0 & 0 & 0 & 0 & 0 & 0 & 1 & 0 & 0 & 0 & 0 & 0 & 0 & 0 & 0\\
 0 & 0 & 0 & 0 & 0 & 0 & 0 & 0 & 0 & 0 & 0 & 1 & 0 & 0 & 0 & 0\\
 0 & 0 & 0 & 0 & 0 & 0 & 0 & 0 & 0 & 0 & 0 & 0 & 0 & 0 & 0 & q\\

  \end{array}
\right)                 
\end{tiny}
\end{equation*}

This matrix as a $k-$linear operator from $V\otimes V$ to $V\otimes V$, with $k=\mathbb{Z}[q,q^{-1}]$ and $V=k\{v_{1},v_{2},v_{3},v_{4}\}$ the free $k-$module generated by four elements, is a Yang-Baxter operator with the standard basis in tensor product of $V\otimes V$.  However, if we divide the elements of each column by the summation of the elements in the corresponding column, it is no longer a Yang-Baxter operator. We have checked this by using Mathematica directly.

\end{counterexample}

\section{Computation of homology for Yang-Baxter operators leading to Homflypt polynomial}\label{Section 5}

In this section, we are interested in the second homology of $R_{(m)}$. First we recall the definition of Yang-Baxter homology for column unital operators.
Let $k$ be a commutative ring, $V=k$X be the free $k-$module generated by basis in $X$, and let the chain modules be $C_{n}(R)=V^{\otimes n}.$  The boundary homomorphism $\partial_{n}:C_{n}(R)\to C_{n-1}(R)$ is given as follows,  $$\partial_{n}=\sum_{i=1}^{n} (-1)^{i}(d^{l}_{i,n}-d^{r}_{i,n}).$$ The face maps $d^{l}_{i,n}$ and $d^{r}_{i,n}$ are illustrated in Figure \ref{face}, where going from top to bottom, and whenever we meet a crossing we apply the Yang-Baxter operator $R$ and we delete the first tensor factor or the last tensor factor at the bottom for $d^{l}_{i,n}$ and $d^{r}_{i,n}$, respectively.  See \cite{Prz,PW} for detail.

\begin{figure}
\includegraphics[scale=1]{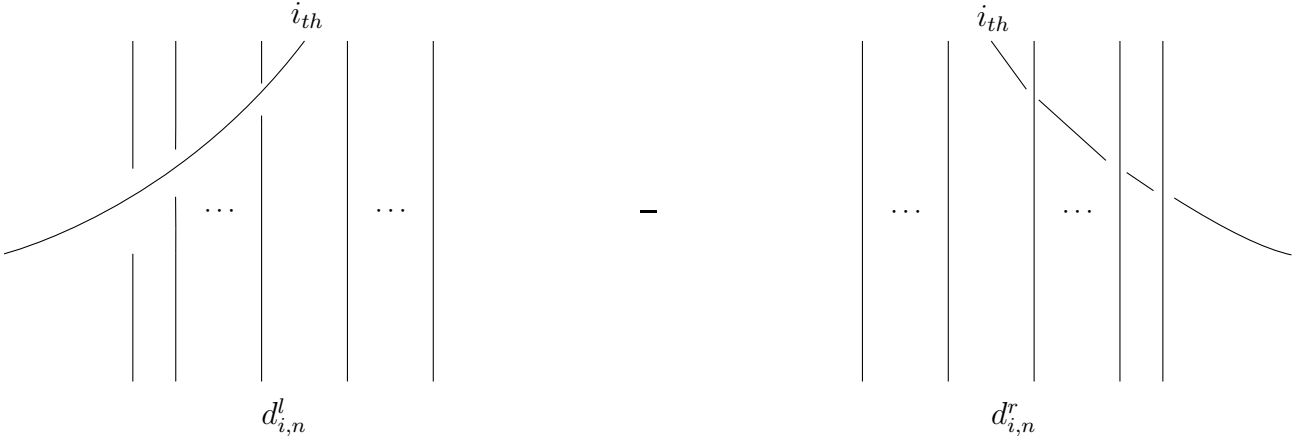} 
\caption{Face maps} \label{face}
\end{figure}

Consider pre-Yang-Baxter operators $R:V \otimes V \to V \otimes V$ given on the basis $X^{2}$ by 
$$R(a,b)=\sum_{c,d} R^{a,b}_{c,d}\cdot (c,d)$$ 
with column unital condition, that is $\sum_{c,d} R^{a,b}_{c,d}=1$ for every $(a,b)\in X^2$. Now $C_1(R)=V$ and $C_2(R)=V^{\otimes 2}$ and
$$\partial_2(a,b)= \sum_{i=1}^2(-1)^i(d^{\ell}_i(a,b) -d^r_i(a,b)) = $$ 
$$(b) - \sum_{c,d}R^{a,b}_{c,d}(d) - \bigg(\sum_{c,d}R^{a,b}_{c,d}(c) - (a)\bigg)= $$
$$(a)+(b) - \sum_{c,d}R^{a,b}_{c,d}((c)+(d))$$
and
$$\partial_3(a,b,c)= \sum_{i=1}^3(-1)^i(d^{\ell}_i(a,b,c) -d^r_i(a,b,c)).$$
Now we go back to analysis of the chain complex for the column unital matrices $R_{(m)}$ in Theorem \ref{Column unital}.  Recall that
$$R^{a b}_{c d}=\left\{
                  \begin{array}{ll}
                    1, & \hbox{if d=a$\geq $b=c;} \\
                    y^2, & \hbox{if d=a$<$b=c;} \\
                    1-y^2, & \hbox{if c=a$<$b=d;} \\
                    0, & \hbox{otherwise.}
                  \end{array}
                \right.$$
In particular, for $m=2$ we have the matrix
\[ R_{(2)}=
\left[
	\begin{array}{cccc}
		1 & 0     & 0 & 0 \\
		0 & 1-y^2 & 1 & 0 \\
		0 & y^2   & 0 & 0 \\
		0 & 0     & 0 & 1 
	\end{array}
	\right],
R_{(2)}^{-1}= \left[
        \begin{array}{cccc}
                1 & 0     & 0        & 0 \\
		0 & 0     & y^{-2}   & 0 \\
		0 & 1     & 1-y^{-2} & 0 \\
                0 & 0     & 0        & 1 
        \end{array}
        \right]
	\]
In Lemma \ref{boundary2}, we prove that the second boundary map is trivial.	
\begin{lemma}\label{boundary2}

For the family of column unital Yang-Baxter operators in Theorem \ref{Column unital}, $\partial_{2}=0$ and $$H_1(R)= C_1(R_{(m)})= V$$ and $$ker\partial_2= C_2(R_{(m)})=V^{\otimes 2}$$

\end{lemma}

\begin{proof}

We check now that $\partial_2(a,b)=0$ for any pair $a,b\in X^2$.
The main reason for $\partial_2=0$ is that if $\{a,b\}\neq \{c,d\}$ then $R^{a,b}_{c,d}=0$ and the column unital property. 
That is for $a,b \in X$:
$$\sum_{c,d}R^{a,b}_{c,d}(c,d)= R^{a,b}_{a,b}(a,b) + R^{a,b}_{b,a}(b,a) \mbox{ with } R^{a,b}_{a,b}+ R^{a,b}_{b,a}=1.$$ 
so $\partial_2(a,b)= (a)+(b) - (R^{a,b}_{a,b}+ R^{a,b}_{b,a})((a)+(b))=0$.\\
Thus 
$$H_1(R_{(m)})= C_1(R_{(m)})=V$$ and $$ker\partial_2= C_2(R_{(m)})=V^2.$$ 
\end{proof}
To compute $H_2(R_{(m)})$, we need to understand $im \partial_3$. The following lemma will be used later in computation,

\begin{lemma}\label{boundary3} \
For the column unital Yang-Baxter operators in Theorem \ref{Column unital}, we have\
\begin{enumerate}

\item  $\partial_{3}(v_{m},a_{1},a_{2})=0$ and $\partial_{3}(a_{1},a_{2},v_{1})=0$, for all $a_{1},a_{2}\in X$, where as before $v_{m}$ is the largest element and $v_{1}$ is the smallest element in $X$;

\item $\partial_{3}(a_{1},a_{2},a_{3})=0$ if either $a_{1}\geq a_{i}$ for all $i=1,2,3$ or $a_{3}\leq a_{j}$ for all $j=1,2,3$, for all $a_{1},a_{2},a_{3}\in X$.

\end{enumerate}

\end{lemma}

\begin{proof}\

Part $(1)$ follows from Lemma \ref{boundary2}.  $\partial_{3}(v_{m},a_{1},a_{2})=[d_{1}^{l}-d_{1}^{r}](v_{m},a_{1},a_{2})-v_{m}\otimes \partial_{2}(a_{1},a_{2})$, by Lemma \ref{boundary2}, $\partial_{3}(v_{m},a_{1},a_{2})=[d_{1}^{l}-d_{1}^{r}](v_{m},a_{1},a_{2})$.  Note that $R(a_{1},a_{2})=(a_{2},a_{1})$ whenever $a_{1}\geq a_{2}$,  $\partial_{3}(v_{m},a_{1},a_{2})=[d_{1}^{l}-d_{1}^{r}](v_{m},a_{1},a_{2})=(a_{1},a_{2})-(a_{1},a_{2})=0$.  Similarly, $\partial_{3}(a_{1},a_{2},v_{1})=0$.\

Part $(2)$ follows from part $(1)$ by considering the subchain complex given by the subspace $\{v_{1},v_{2},...,a_{1}\}$ of $V_{m}$ or $\{a_{3}...,a_{m-1},a_{m}\}$ of $V_{m}$ respectively.

\end{proof}

\ \\

The main result of this section is as follows.  Notice that the ring $k$ can be either $\mathbb{Z}[y^{\pm}]$ or $\mathbb{Z}[y]$.

\begin{theorem}\label{Theorem 5.1} Let $R$ be a unital Yang-Baxter operator giving 
	Homflypt polynomial on level $m$ in Theorem \ref{Column unital}, then 
	$$H_2(R)= k^{1+{m \choose 2}} \oplus \bigg(k/(1-y^2)\bigg)^{m \choose 2} \oplus \bigg(k/(1-y^4)\bigg)^{m-1}.$$
\end{theorem}
\begin{proof}

First, we compute $\partial_{3}$.\
Let $a<b<c$, we need to consider $13$ cases, which are:$$(a,b,c);(b,c,a);(c,a,b);(b,a,c);(a,c,b);(c,b,a);(a,a,b);(b,b,a);(a,b,b);(b,a,a);(a,b,a);(b,a,b);(a,a,a)$$
By Lemma \ref{boundary3}, we have $\partial_3(b,c,a)=\partial_3(c,a,b)=\partial_3(c,b,a)=\partial_{3}(b,b,a)=\partial_{3}(b,a,a)=\partial_{3}(a,b,a)=\partial_{3}(b,a,b)=\partial_{3}(a,a,a)=0$.

$\partial_{3}$ provides non-trivial relations in the homology for the $5$ remaining cases (however, they are not all linearly independent). Let us demonstrate the calculation of $\partial_{3}(a,b,c)$.

We make calculation easy by considering graphical interpretation of face maps $d_i^{\epsilon}$, starting from the defining formula:
$$\partial_3=d^{\ell}_1+d_2^r+d_3^{\ell}-\bigg(d^r_3+d_2^{\ell}+d_1^r\bigg).$$

\center{
$\partial_3(a,b,c)=
{\parbox{1.2cm}{\psfig{figure=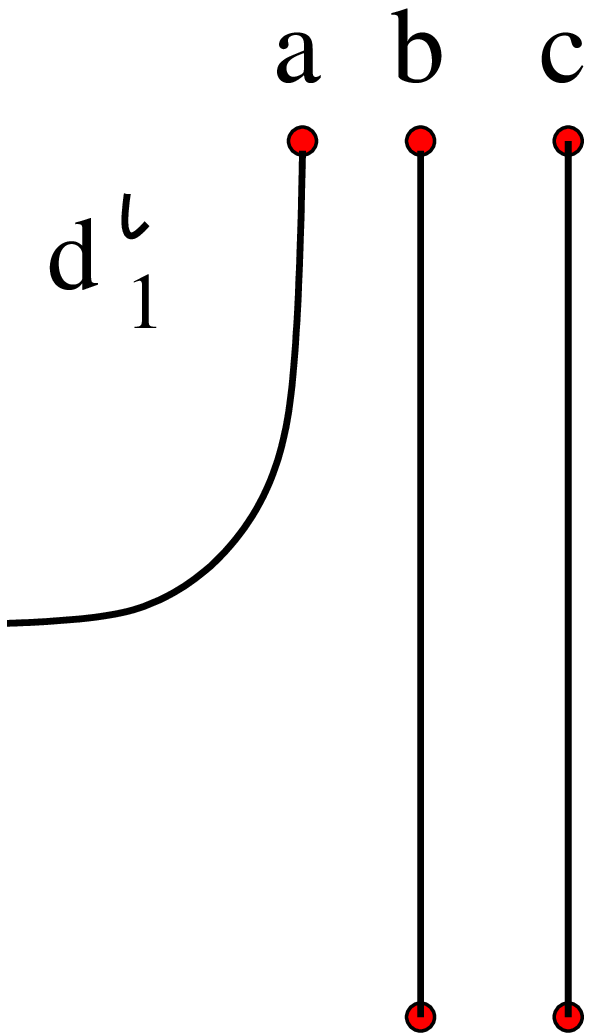,height=1.7cm}}}+ {\parbox{1.2cm}{\psfig{figure=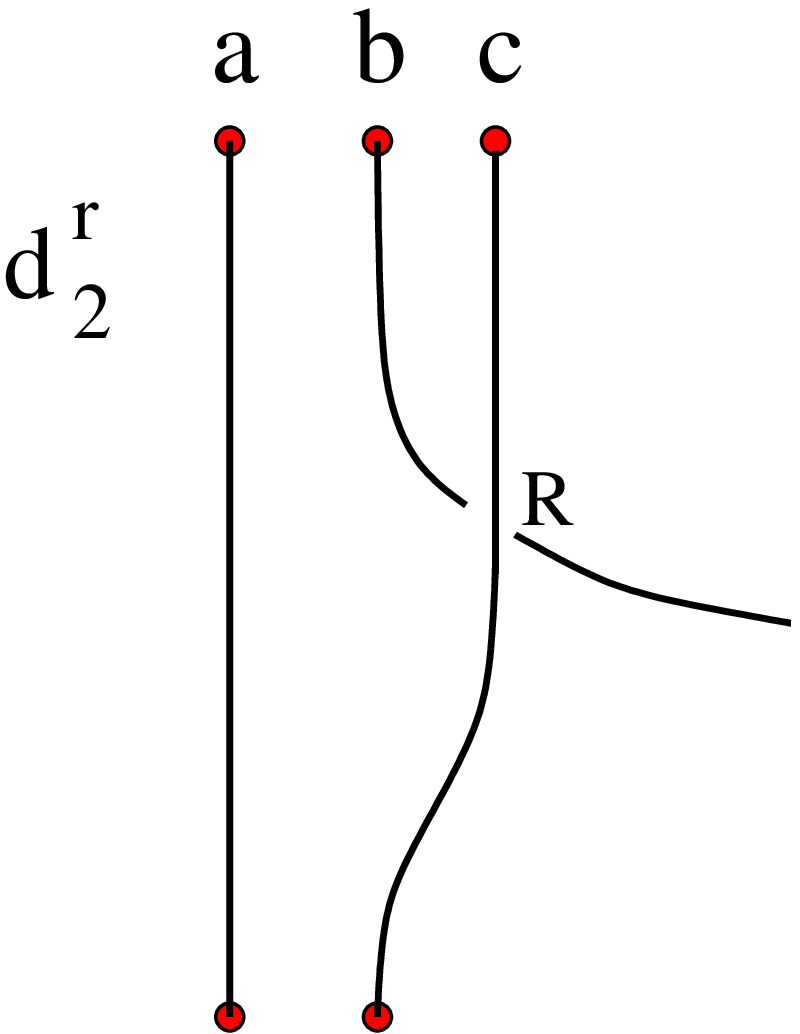,height=1.7cm}}}+
{\parbox{1.2cm}{\psfig{figure=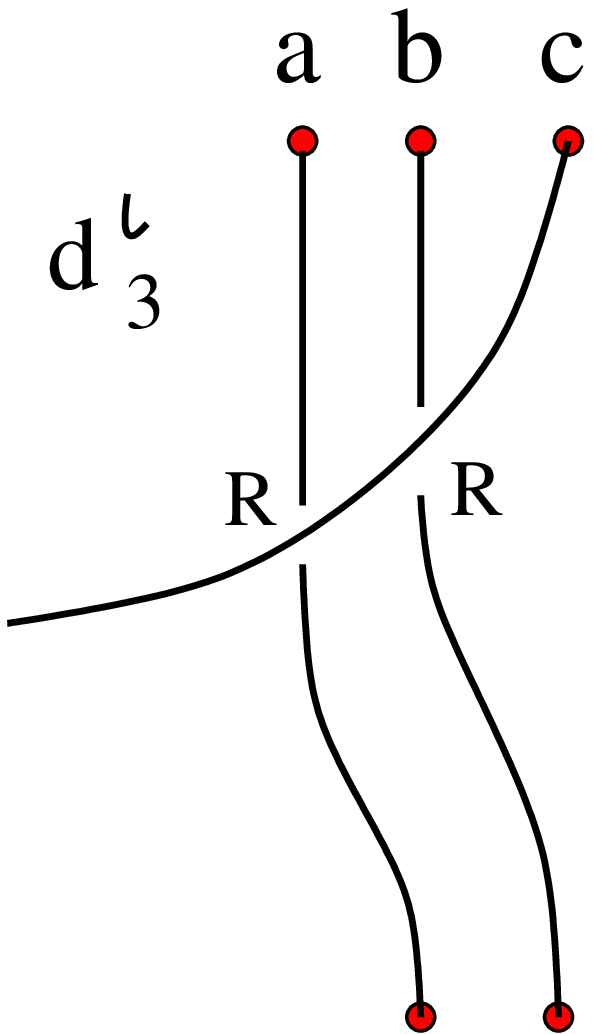,height=1.7cm}}}- {\parbox{1.2cm}{\psfig{figure=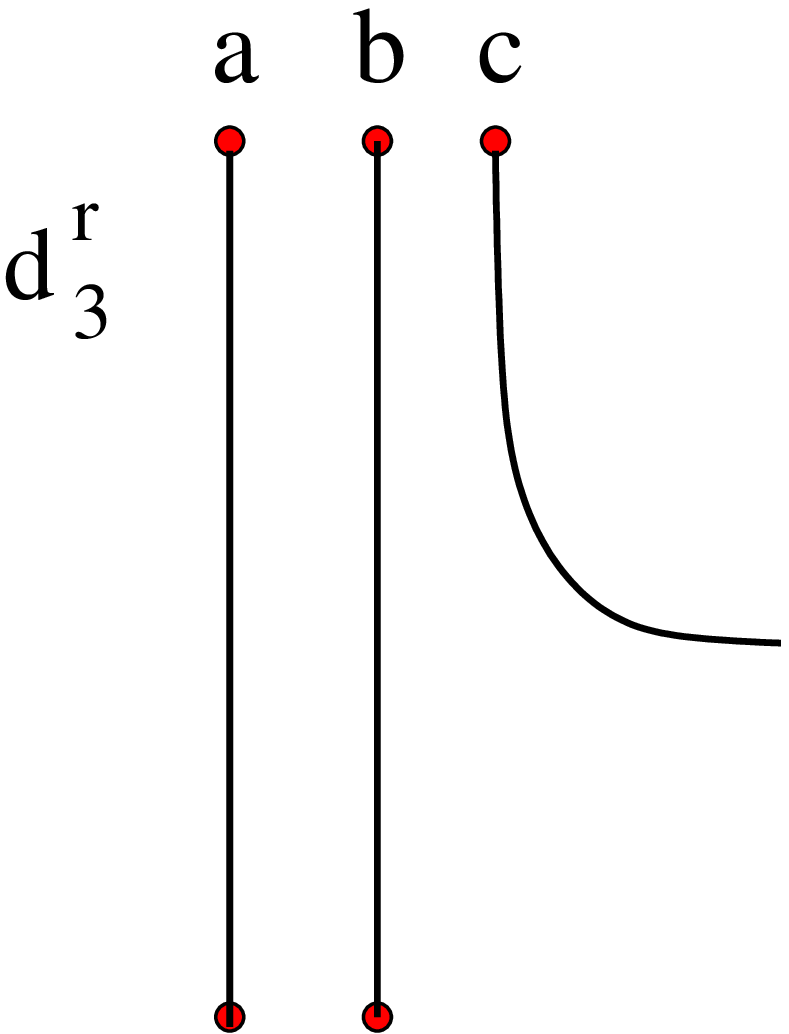,height=1.7cm}}}-
{\parbox{1.2cm}{\psfig{figure=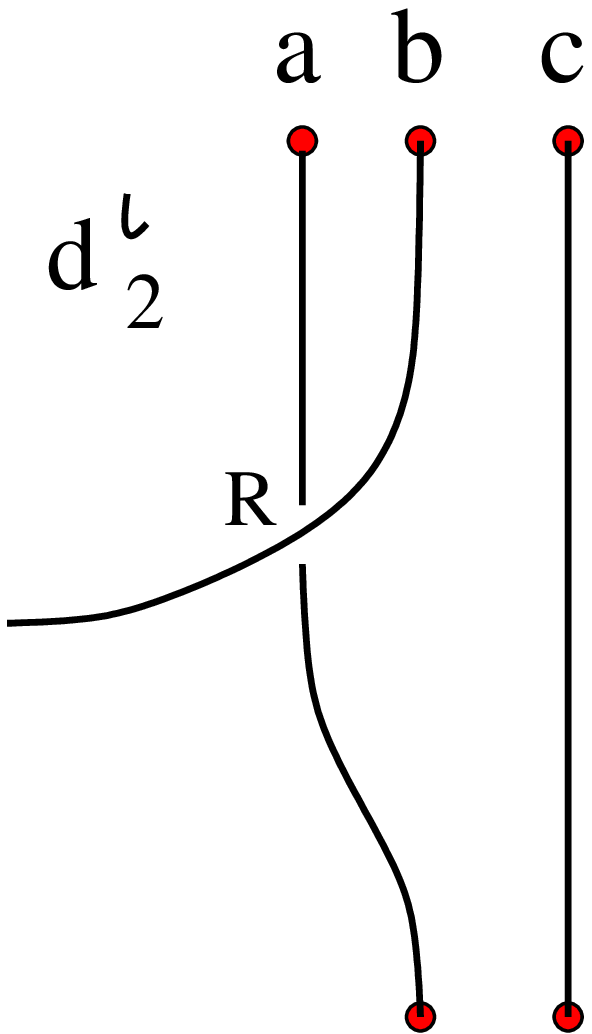,height=1.7cm}}}-{\parbox{1.2cm}{\psfig{figure=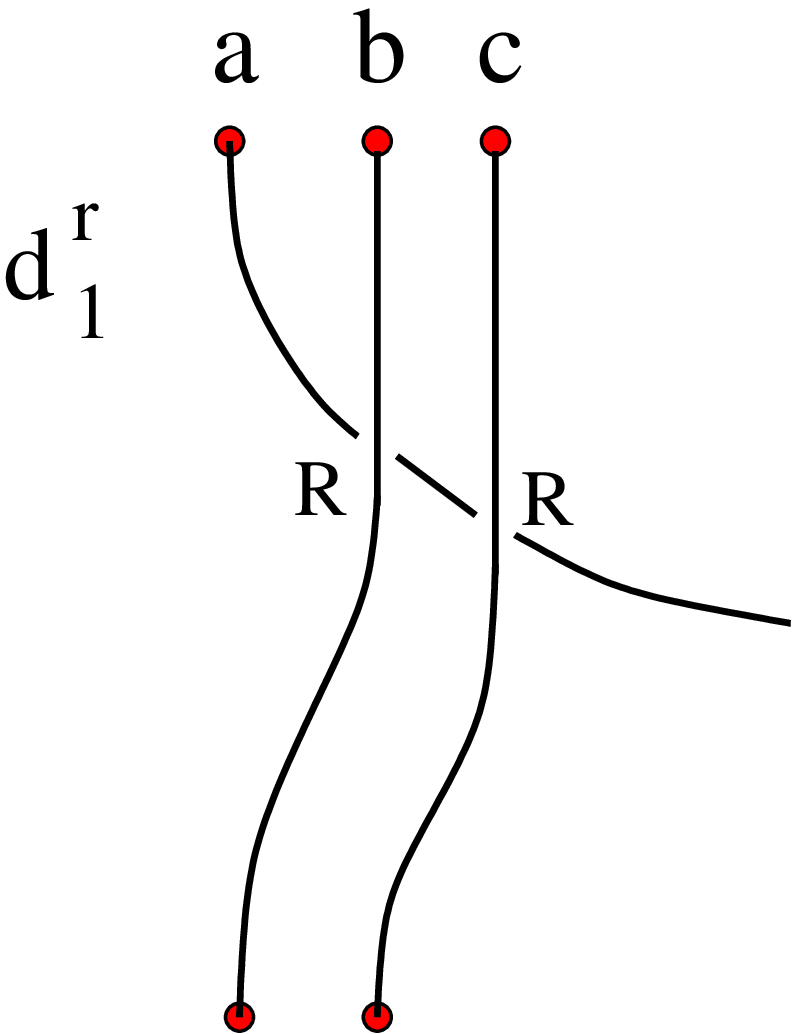,height=1.7cm}}}$}\\

From this diagrams we compute (keeping the terms in the same order as in figures):
$$\partial_3(a,b,c)= (b,c)+\big(y^2(a,c)+ (1-y^2)(ab)\big)+$$
$$\bigg(y^4(a,b)+y^2(1-y^2)(c,b)+(1-y^2)y^2(a,c)+(1-y^2)^2(b,c)\bigg) - $$
$$(a,b)-\big(y^2(a,c)+ (1-y^2)(b,c)\big) -$$
$$\bigg(y^4(b,c) +y^2(1-y^2)(b,a)+(1-y^2)y^2(a,c)+(1-y^2)^2(a,b)\bigg)=$$
$$(1-y^2)\bigg((b,c)-(a,b) +y^2((c,b)-(b,a))\bigg).$$
The longest calculation is that of $d_3^{\ell}$ and $d_1^{r}$. In the next picture we illustrate how to compute 
quickly $d_3^{\ell}$:

\centerline{\psfig{figure=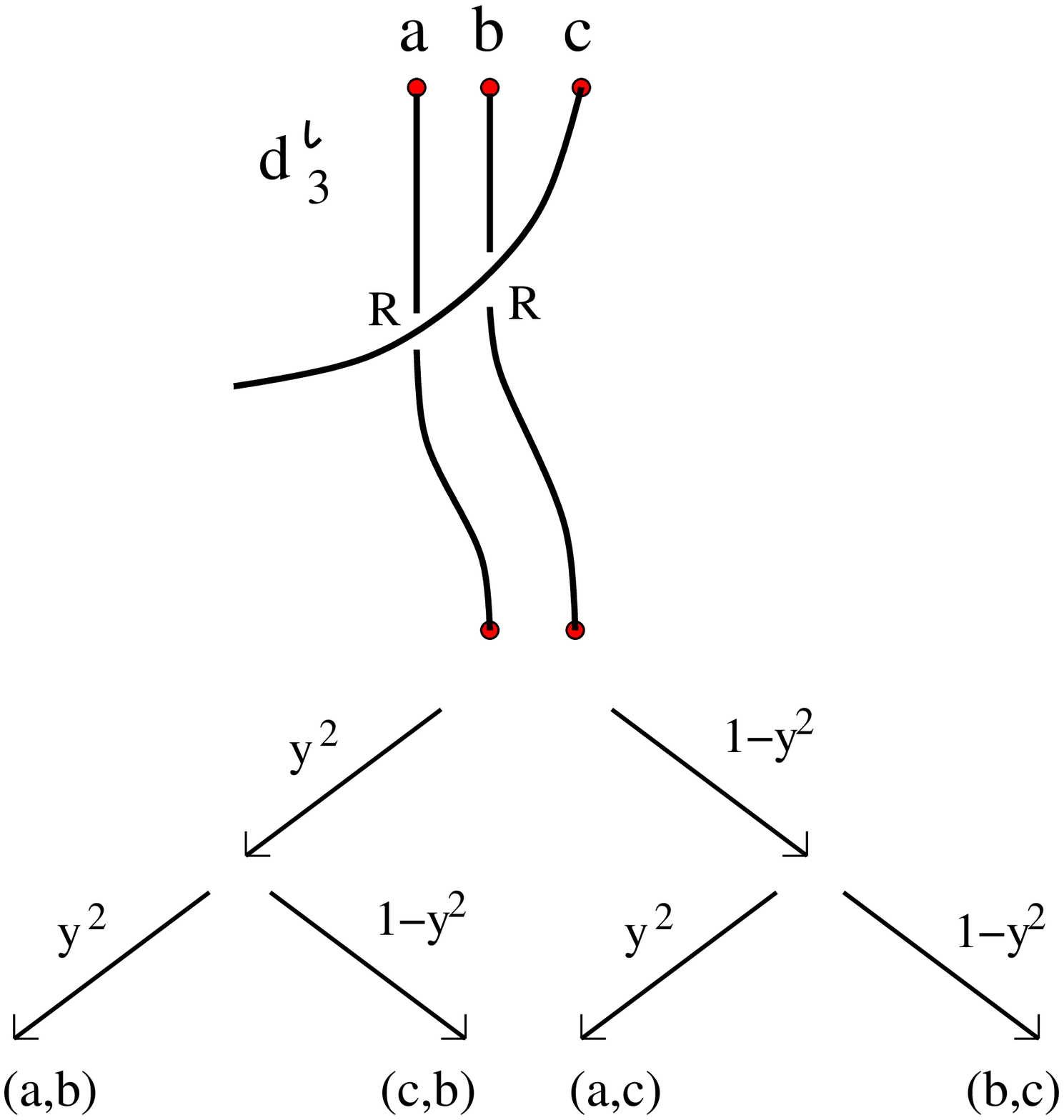,height=13.3cm}\label{tree}}
\centerline{ Figure \ref{tree} $d_3^{\ell}=y^4(a,b)+y^2(1-y^2)(c,b)+(1-y^2)y^2(a,c)+(1-y^2)^2(b,c)$}
\ \\
\ \\
The computation for $d_1^r$ is similar. We can also use the symmetry that is $a$ and $c$ switch roles and $(x,y)$ goes to $(y,x)$.
Thus we get
$d_1^{r}=y^4(b,c)+y^2(1-y^2)(b,a)+(1-y^2)y^2(a,c)+(1-y^2)^2(a,b).$

\ \\
With some efforts, we get the following non-trivial differentials of elements with three distinct letters:
$$\partial_3(a,b,c)=(1-y^2)\bigg((b,c)-(a,b) +y^2((c,b)-(b,a))\bigg).$$
$$\partial_3(a,c,b)=(1-y^2)\bigg((b,c)-(a,c) +y^2((c,b)-(c,a))\bigg).$$
$$\partial_3(b,a,c)=(1-y^2)\bigg((a,c)-(a,b) +y^2((c,a)-(b,a))\bigg).$$
They are not independent as:
$$ \partial_3(a,b,c) - \partial_3(a,c,b) - \partial_3(b,a,c)=0.$$

Also, by Proposition \ref{leq}, we have:
$$\partial_3(a,a,b)=(1-y^2)\bigg((a,b)-(a,a) +y^2((b,a)-(a,a))\bigg).$$
$$\partial_3(a,b,b)=(1-y^2)\bigg((b,b)-(a,b) +y^2((b,b)-(b,a))\bigg).$$
From the following two equations, we see that the relations given by $\partial_{3}$ are generated by the images of $(a,a,b)$ and $(a,b,b)$ as follows.
$$\partial_3(a,b,c)= (1-y^2)\bigg((b,c)-(a,b) +y^2((c,b)-(b,a))\bigg)= \partial_3(b,b,c)+ \partial_3(a,b,b),$$
and 
$$\partial_3(b,a,c)= (1-y^2)\bigg((a,c)-(a,b) +y^2((c,a)-(b,a))\bigg)= \partial_3(a,a,c) - \partial_3(a,a,b).$$
\ \\
Let us summarize the structure of the image $\partial_3(C_3)$. It is generated by 
$$\partial_3(v_i,v_i,v_j)=(1-y^2)\big((v_i,v_j)- (v_i,v_i)+y^2((v_j,v_i)-(v_i,v_i)\big) \mbox{ for $i<j$},$$
and 
$$\partial_3\big((v_i,v_i,v_j)+(v_i,v_j,v_j)\big)= (1-y^4)\big((v_j,v_j)-(v_i,v_i)\big) \mbox{ for $i<j$}.$$
We notice quickly that $\partial_3\big((v_i,v_i,v_j)+(v_i,v_j,v_j)\big)$ is generated by $m-1$ elements $(v_j,v_j)-(v_1,v_1)$ with $m\geq j >1$. \\

	Consider the following new basis of $kX^{2}$ consisting of three groups of basis elements:
	$$X_0=\{(v_1,v_1), (v_j,v_i) \mbox{ for $i<j$}\} \mbox{ that is ${m \choose 2}+1$ elements}.$$
	$$X_1=\{(v_i,v_j)- (v_i,v_i)+y^2((v_j,v_i)-(v_i,v_i)) \mbox{ for $i<j$}\} \mbox{ that is ${m \choose 2}$ elements}.$$
	$$X_2=\{(v_j,v_j)-(v_1,v_1) \mbox{ with $m\geq j >1$}\} \mbox{ that is $m-1$ elements} .$$
	Clearly $X_0\sqcup X_1 \sqcup X_2$ form a basis of $kX^2$.\\
	We look now at relations: in our basis, the matrix of relations is diagonal with $0$ for elements in $X_0$, $1-y^2$ for elements in $X_1$, and $(1-y^4)$ for elements in $X_2$. Thus not only we proved that 
$$H_2(R)= k^{1+{m \choose 2}} \oplus \bigg(k/(1-y^2)\bigg)^{m \choose 2} \oplus \bigg(k/(1-y^4)\bigg)^{m-1}.$$
but we also found a basis of $C_2= kX^2$ realizing the decomposition into cyclic submodules.
\end{proof}

From Theorem \ref{Theorem 5.1}, we can easily see the rank of $ker \partial_{3}(R_{(m)})$.

\begin{corollary}\label{ker}

$Rank(ker \partial_{3}(R_{(m)}))=\frac{(m+1)(2m^{2}-3m+2)}{2}.$

\end{corollary}

\begin{proof}

Rank of the kernal $\partial_{3}$ is the rank of $C_{3}$ minus the number of nonzero elements in the diagonal relation matrix of $\partial_{3}$, which is exactly the numbers of $(1-y^{2})$ and $(1-y^{4})$. Thus $$Rank(ker \partial_{3}(R_{(m)}))=m^{3}-{m \choose 2}-(m-1)=\frac{(m+1)(2m^{2}-3m+2)}{2}$$

\end{proof}

\section{Further computations and future work}

Here we summarise all data obtained with the help of computer. Because of the limitation of the computation program, the computation were done over the ring $\mathbb{Q}[y]$. In \cite{PW}, we formulated a conjecture about the homology of $R_{(m)}$ when $m=2$ as follows, 
\begin{conjecture}\label{PW}\cite{PW}
When $m=2,$ $H_{n}=k^{2}\bigoplus (k/(1-y^{2}))^{a_{n}}\bigoplus (k/(1-y^{4}))^{s_{n-2}},$ where $s_{n}=\Sigma_{i=1}^{n+1}f_{i}$ is the partial sum of Fibonacci sequence, where $f_{1}=1=f_{2}$ and $a_{n}$ is given by $2^{n}=2+a_{n-1}+s_{n-3}+a_{n}+s_{n-2}$ with $a_{1}=0.$
\end{conjecture}
This conjecture is verified upto $n\leq 11$ by computer.  In the paper \cite{ESZ}, there is a discussion of various aspects of this conjecture. More computation is shown in Table \ref{Data}, where $(x,y,z)$ represents decomposition into $x$ copies of $k$, $y$ copies of $k/(1-y^{2})$ torsion and $z$ copies of $k/(1-y^{4})$ torsion.

\begin{center}
 \begin{tabular}{||c c c c c c||} 
 \hline
 $H_{n}$ & m=3 & m=4 & m=5 & m=6 & m=7 \\ [0.5ex] 
 \hline
 $H_{2}$ & (4,3,2) & (7,6,3) & (11,10,4) & (16,15,5) & (22,21,6)\\ 
 \hline
 $H_{3}$ & (4,12,6) & (8,35,12) & (15,76,20) & (26,140,30) & (42,232,42) \\ [1ex] 
 \hline
\end{tabular}\label{Data}
\centerline{Table \ref{Data}}
\end{center}

From the first row of the table, we can see that the results match with that of the formula in Theorem \ref{Theorem 5.1}.  From the second row of the table, we conjecture the formula for $H_{3}$ as follows,

\begin{conjecture}\label{H_{3}}

$H_{3}(R_{(m)})=k^{\frac{m(8-3m+m^{2})}{6}}\bigoplus (k/(1-y^{2}))^{\frac{(m^{2}-1)(5m-6)}{6}}\bigoplus (k/(1-y^{4}))^{m(m-1)}$

\end{conjecture}

\begin{remark}\

\begin{enumerate}

\item The ranks in Conjecture \ref{H_{3}} sum up to the rank of $ker \partial_{3}$.

\item The rank of $H_3(R_{(2)})$ in Conjecture \ref{PW} agrees with the rank of $H_{3}(R_{(2)})$ in Conjecture \ref{H_{3}}.

\end{enumerate}

\end{remark}

Computations and patterns observed so far suggest that there are only two types of torsion elements $k/(1-y^{2})$ and $k/(1-y^{4})$. However, this is only checked upto the strength of computer program.  By analyzing the boundary maps in general, we hope to gain more information about $H_{n}(R_{(m)})$.  The first step towards this goal is the following observation.

\begin{remark}\label{1-y^{2}}

The factor $(1-y^{2})$ divides every element in Im($\partial_{n}$).  This follows from the fact that when setting $1-y^{2}=0$, $d^{\ell}_{i}=d^{r}_{i}$.  Thus, we have $\partial_{n}(a_{1},...,a_{n})\subset (1-y^{2})V^{n}$, where $a_{i}\in X_{m}$, $i=1,2,...,m$.  One possible approach to compute $H_{n}(R_{(m)})$ is to decompose the boundary map along the factors $(1-y^{2})^{i}$. In the first step, we ignore the branches with factor $(1-y^{2})$ in the computational tree, see Figure \ref{tree}.  Generally, in the $i-th$ step, we ignore the paths in the computational tree which are going $i$ or more times to the right.

\end{remark}
\ \\

\section{Acknowledgements}
J.~H.~Przytycki was partially supported by the Simons Collaboration Grant-637794 and GWU, CCAS Enhanced Travel Award.  X. Wang was supported by the National Natural Science Foundation of China (No.11901229).

\end{document}